\documentclass[a4paper,12pt,reqno]{amsart}
\usepackage{amssymb}
\usepackage{amsmath}
\usepackage{ifthen}
\usepackage[dvips]{graphicx}
\nonstopmode \numberwithin{equation}{section}
\usepackage{amssymb}
\usepackage{ifthen}

\usepackage{amsmath}
\usepackage[T1]{fontenc} 
\usepackage{comment}

\nonstopmode \numberwithin{equation}{section}
\theoremstyle{plain}
\newtheorem{thm}{Theorem}
\numberwithin{thm}{section}
\newtheorem{cor}{Corollary}
\numberwithin{cor}{section}
\newtheorem{lem}{Lemma}
\numberwithin{lem}{section}
\newtheorem{prop}{Proposition}

\newtheorem{conj}{Conjecture}

\newcommand\numberthis{\addtocounter{equation}{1}\tag{\theequation}}
\theoremstyle{definition}
\newtheorem{defn}{Definition}[section]

\newtheorem{prob}{Problem}
\newtheorem{rem}{Remark}[section]

\newcounter{minutes}
\divide\time by 60
\newcounter{hours}
\multiply\time by 60
\addtocounter{minutes}{-\time}

\newcounter {own}
\def\theown {\thesection       .\arabic{own}}

\newenvironment{pf}[1][]{%
	\vskip 3mm
	\noindent
	\ifthenelse{\equal{#1}{}}%
	{{\slshape Proof. }}%
	{{\slshape #1.} }%
}%
{\qed\bigskip}

\theoremstyle{plain}
\newtheorem{Thm}{Theorem}

\newcommand{\ID}{{\mathbb D}}

\newcommand{\IC}{{\mathbb C}}

\renewcommand{\theequation}{\thesection.
\arabic{equation}}
\numberwithin{equation}{section}
\def\be{\begin{equation}}
\def\ee{\end{equation}}

\newcommand{\bee}{\begin{enumerate}}
	\newcommand{\eee}{\end{enumerate}}

\newcommand{\blem}{\begin{lem}}
	\newcommand{\elem}{\end{lem}}
\newcommand{\bthm}{\begin{thm}}
	\newcommand{\ethm}{\end{thm}}
\newcommand{\bcor}{\begin{cor}}
	\newcommand{\ecor}{\end{cor}}
\newcommand{\beg}{\begin{examp}}
	\newcommand{\eeg}{\end{examp}}
\newcommand{\begs}{\begin{examples}}
	\newcommand{\eegs}{\end{examples}}
\allowdisplaybreaks
\newcommand{\bdefn}{\begin{defn}}
	\newcommand{\edefn}{\end{defn}}

\newcommand{\bprob}{\begin{prob}}
	\newcommand{\eprob}{\end{prob}}
\newcommand{\bei}{\begin{itemize}}
	\newcommand{\eei}{\end{itemize}}

\newcommand{\bcon}{\begin{conj}}
	\newcommand{\econ}{\end{conj}}
\newcommand{\bcons}{\begin{conjs}}
	\newcommand{\econs}{\end{conjs}}
\newcommand{\bprop}{\begin{prop}}
	\newcommand{\eprop}{\end{prop}}
\newcommand{\br}{\begin{rem}}
	\newcommand{\er}{\end{rem}}
\newcommand{\brs}{\begin{rems}}
	\newcommand{\ers}{\end{rems}}
\newcommand{\bo}{\begin{obser}}
	\newcommand{\eo}{\end{obser}}
\newcommand{\bos}{\begin{obsers}}
	\newcommand{\eos}{\end{obsers}}
\newcommand{\bpf}{\begin{pf}}
	\newcommand{\epf}{\end{pf}}
\newcommand{\ba}{\begin{array}}
	\newcommand{\ea}{\end{array}}
\newcommand{\beq}{\begin{eqnarray}}
\newcommand{\beqq}{\begin{eqnarray*}}
\newcommand{\eeq}{\end{eqnarray}}
\newcommand{\eeqq}{\end{eqnarray*}}

\begin{document}
\title{On Landau-Type Theorems for Poly-Analytic Functions}

\author{Vasudevarao Allu}
\address{Vasudevarao Allu,
	Department of Mathematics, School of Basic Science,
	Indian Institute of Technology Bhubaneswar,
	Bhubaneswar-752050, Odisha, India.}
\email{avrao@iitbbs.ac.in}

\author{Rohit Kumar}
\address{Rohit Kumar,
	Department of Mathematics, School of Basic Science,
	Indian Institute of Technology Bhubaneswar,
	Bhubaneswar-752050, Odisha, India.}
\email{rohitk12798@gmail.com}

\subjclass[{AMS} Subject Classification:]{Primary 31A30, 30C99; Secondary 31A05, 31A35.}

\def\thefootnote{}
\footnotetext{ {\tiny File:~\jobname.tex,
		printed: \number\year-\number\month-\number\day,
		\thehours.\ifnum\theminutes<10{0}\fi\theminutes }
} \makeatletter\def\thefootnote{\@arabic\c@footnote}\makeatother

\begin{abstract}
In this paper, we establish three Landau-type theorems for certain bounded poly-analytic functions, which generalize the corresponding result for bi-analytic functions given by Liu and Ponnusamy [Canad. Math. Bull. 67(1): 2024, 152-165]. Further, we prove three bi-Lipschitz theorems for these subclasses of poly-analytic functions.

\end{abstract}

\maketitle
\pagestyle{myheadings}
\markboth{Vasudevarao Allu and  Rohit Kumar}{On Landau-Type Theorems for Poly-Analytic Functions}

\section{\textbf{Introduction and Preliminaries}}
There are non-analytic functions that satisfy properties similar to those of analytic functions, and they have a notable structure. Such non-analytic functions are called poly-analytic functions. 
A complex-valued function $F: \Omega \subset \mathbb{C} \rightarrow \mathbb{C}$ of class $C^m$ on a domain $\Omega$ is called a poly-analytic function of order $ m \geq 1$, if
$$
\frac{\partial^m}{\partial \bar{z}^m} F(z)=0, \quad \text{for all } z \in \Omega.
$$
 Any poly-analytic function of order $m$ can be decomposed in terms of $m$ analytic functions so that we have a decomposition of the following form
$$
F(z)=\sum_{k=0}^{m-1} \bar{z}^k f_k(z).
$$
For a continuously differentiable function $F$, denote
$$
\Lambda_F=\max _{0 \leq \theta \leq 2 \pi}\left|F_z+e^{-2 i \theta} F_{\bar{z}}\right|=\left|F_z\right|+\left|F_{\bar{z}}\right|$$ and
$$
\lambda_F=\min _{0 \leq \theta \leq 2 \pi}\left|F_z+e^{-2 i \theta} F_{\bar{z}}\right|=|| F_z|-| F_{\bar{z}}||,
$$
where $F_z=\partial F / \partial z$ and $F_{\bar{z}}=\partial F / \partial \bar{z}$. A mapping $g: \Omega \rightarrow \IC$ is said to be $L_1$-Lipschitz $\left(L_1>0\right)$ if
$$
\left|g\left(z_1\right)-g\left(z_2\right)\right| \leq L_1\left|z_1-z_2\right|, z_1, z_2 \in \Omega
$$
and it is said to be $l_1$-co-Lipschitz $\left(l_1>0\right)$ if
$$
\left|g\left(z_1\right)-g\left(z_2\right)\right| \geq l_1\left|z_1-z_2\right|, z_1, z_2 \in \Omega.
$$
A mapping $\omega$ is bi-Lipschitz if it is Lipschitz and co-Lipschitz.

\vspace{2mm}
The theory of poly-analytic functions is an interesting topic in complex analysis. It extends the concept of holomorphic functions to nullsolutions of higher-order powers of the Cauchy-Riemann operator. These functions were first introduced in 1908 by Kolossov \cite{Kolossov-1908} to study elasticity problems. For a complete introduction to poly-analytic functions and their basic properties we refer to \cite{Balk-1991,Balk-1997}. The class of poly-analytic functions have been studied by various authors from different perspectives; see \cite{Abreu-2010,Agranovsky-2011,Ahern-Bruna-1988,Vasilevski-1999} and the references therein. Recently, the Landau-type theorem for some subclasses of poly-analytic functions has been studied by Abdulhadi and Hajj \cite{Abdulhadi-Hajj-2022} and Liu and Ponnusamy \cite{Liu-Ponnusamy-2024}. \\[2mm]
The study of polyanalytic functions has seen substantial progress over recent years, with foundational and applied insights emerging in various works. In 2016, Daghighi and Krantz \cite{Daghighi-Krantz-2016} explored the local maximum modulus property for polyanalytic functions, extending the classical principle from analytic to polyanalytic settings and shedding light on the geometric and functional behaviour of these functions. Building on this, in 2023 Allu and Halder \cite{Allu-Halder-2023} explored the Bohr operator on operator-valued polyanalytic functions in simply connected domains, providing operator-theoretic insights. Most recently, Vasilevski \cite{Vasilevski-2023} examined their properties in several complex variables, extending classical analysis.

\vspace{2mm}
 Let $\mathbb{D}=\{z\in\mathbb{C}:|z|<1\}$ represent the unit disc in the complex plane $\mathbb{C}$. Throughout this paper, $\ID_r=\{z\in \IC:|z|<r\}$ denotes the open disc about the origin. For the convenience of the reader, let us fix some basic notations. 
\begin{itemize}
\item $\mathcal{H}_m(\ID)=\{F: F$ is poly-analytic in $\ID$ with $F(z)=\sum_{k=0}^{m-1} \bar{z}^k f_k(z)$, where $f_k$ are analytic in $\ID$\}.
\item $\mathcal{H}_m^0(\ID)= \{F \in \mathcal{H}_m(\ID):F_z(0)=1$ and $f_k(0)=0$ for each $k$\}.
\item $\mathcal{F}_1 = \{F \in \mathcal{H}_m^0(\ID): |f_0'(z)|< \Lambda, |f_k(z)|\leq M_k$ for each $k$\}. 
\item $\mathcal{F}_2=\{F \in \mathcal{H}_m^0(\ID): |f_0(z)|< M, |f_k'(z)|\leq \Lambda_k$ for each $k$\}.
\item $\mathcal{F}_3=\{F \in \mathcal{H}_m^0(\ID): |f_0'(z)|< \Lambda_0, |f_k'(z)|\leq \Lambda_k$ for each $k$\}.
\end{itemize}
\vspace{2mm}

In this paper, we consider the subclasses $\mathcal{F}_1, \mathcal{F}_2$ and $\mathcal{F}_3$ of poly-analytic functions and establish sharp versions of the Landau-type theorem for these classes of mappings. Further, we prove three bi-Lipschitz properties for these classes of functions.

\subsection{The Landau Theorem} 

The classical Landau theorem (see \cite{Landau-1926}) asserts that if $f$ is a holomorphic mapping of the unit disc $\ID=\{z\in \IC:|z|<1\}$ with $f(0)=0$, $f'(0)=1$ and $|f(z)|<M$ for $z \in \mathbb{D}$, then $f$ is univalent in $\mathbb{D}_{r_0}$, and $f(\mathbb{D}_{r_0})$ contains a disc $\mathbb{D}_{\sigma_0}$, where $$r_0=\frac{1}{M+\sqrt{M^2-1}}\ \ \text{and} \ \ \sigma_0= Mr_0^2.$$
  The quantities $r_0$ and $\sigma_0$ cannot be improved as the function $f_0(z)= Mz \left(\frac{1-Mz}{M-z}\right)$ shows the sharpness of the quantities $r_0$ and $\sigma_0$. 
\vspace{2mm}

In 2000, Chen {\it{et al.}} \cite{Chen-Gauthier-Hengertner-2000} obtained the Landau-type theorems for harmonic mappings in the unit disc $\ID$. 

\begin{Thm}\cite{Chen-Gauthier-Hengertner-2000}\label{Thm-A}
Let $f$ be a harmonic mapping of the unit disc $\mathbb{D}$ such that $f(0)=0$, $f_{\bar{z}}(0)=0, f_z(0)=1$, and $|f(z)|<M$ for $z \in \mathbb{D}$. Then, $f$ is univalent on a disc $\mathbb{D}_{\rho_0}$ with
$$
\rho_0=\frac{\pi^2}{16 m M},
$$
and $f\left(\mathbb{D}_{\rho_0}\right)$ contains a schlicht disc $\mathbb{D}_{R_0}$ with
$$
R_0=\rho_0 / 2=\frac{\pi^2}{32 m M},
$$
where $m \approx 6.85$ is the minimum of the function $(3-r^2)/(r(1-r^2))$ for $0<r<1$.
\end{Thm}

\begin{Thm}\cite{Chen-Gauthier-Hengertner-2000}\label{Thm-B}
 Let $f$ be a harmonic mapping of the unit disc $\mathbb{D}$ such that $f(0)=0$, $\lambda_f(0)=1$ and $\Lambda_f(z) \leq \Lambda$ for $z \in \mathbb{D}$. Then, $f$ is univalent on a disc $\mathbb{D}_{\rho_0}$ with
$$
\rho_0=\frac{\pi}{4(1+\Lambda)},
$$
and $f\left(\mathbb{D}_{\rho_0}\right)$ contains a schlicht disc $\mathbb{D}_{R_0}$ with
$$
R_0=\frac{1}{2} \rho_0=\frac{\pi}{8(1+\Lambda)} .
$$
\end{Thm}
It is important to note that the radii $\rho_o$ and $R_0$ in Theorems \ref{Thm-A} and \ref{Thm-B} are not sharp. Theorem \ref{Thm-A} and Theorem \ref{Thm-B} have been improved by Chen {\it{et al.}} \cite{Chen-Ponnusamy-Wang-2011}, Dorff and Nowak \cite{Dorff-Nowak-2004}, Grigoryan \cite{Grigoyan-2006}, Huang \cite{Huang-2008}, Liu \cite{M. Liu-2009, Liu-2009}, Liu and Chen \cite{Liu-Chen-2018} and Zhu \cite{Zhu-2015}. In particular, Liu \cite{Liu-2009} has proved that under the hypothesis of Theorem \ref{Thm-B}, $\Lambda\geq 1$, and when $\Lambda=1$, $f$ is univalent in the disc $\ID_{\rho}$ and $f(\ID_{\rho})$ contains a schlicht disc $\ID_R$ with $\rho = R = 1$ being sharp. In 2018, Liu and Chen \cite{Liu-Chen-2018} established the sharp version of Theorem \ref{Thm-B} for $\Lambda>1$ by applying certain geometric methods. The Landau type theorem has also been studied for elliptic harmonic mappings \cite{Allu-Kumar-2024}, pluriharmonic mappings \cite{Allu-Kumar-2024-a}, $\alpha$-harmonic mappings \cite{Allu-Kumar-2024-b} and logharmonic mappings \cite{Abdulhadi-Muhanna-Ali-2012}.
\vspace{2mm}

The Landau-type theorem for different classes of poly-harmonic and log-p-harmonic has been studied in \cite{Bai-Liu-2019,Liu-Luo-2021,Luo-Liu-2023}. In 2022, the Landau-type theorem for poly-analytic functions was first studied by Abdulhadi and Hajj \cite{Abdulhadi-Hajj-2022}.

\begin{Thm}\cite{Abdulhadi-Hajj-2022}\label{Thm-C}
Let $F(z)=\sum_{k=0}^{m-1} \bar{z}^k f_k(z)$ be a poly-analytic function of order $m, m \geq 2$ on the unit disc $\ID$, where $f_k$ are holomorphic such that $f_k(0)=0, f_k'(0)=1$ and $\left|f_k(z)\right| \leq M$, for all $k$, with $M>1$. Then there is a constant $0<\rho<1$ so that $F$ is univalent in $|z|<\rho$. In particular, $\rho$ satisfies
$$
1-M\left(\frac{\rho\left(2-\rho\right)}{\left(1-\rho\right)^2}+\sum_{k=1}^{m-1} \frac{\rho^k\left(1+k-k \rho\right)}{\left(1-k \rho\right)^2}\right)=0
$$
and $F\left(\ID_{\rho}\right)$ contains a disc $\ID_{\sigma}$, where
$$
\sigma=\rho-\rho^2\left(\frac{1-\rho^{m-1}}{1-\rho}\right)-M \sum_{k=0}^{m-1} \frac{\rho^{k+2}}{1-\rho}.
$$
It is important to note that the radii $\rho$ and $\sigma$ in Theorem \ref{Thm-C} is not sharp.
\end{Thm}

\vspace{2mm}
In 2024, Liu and Ponnusamy \cite{Liu-Ponnusamy-2024} obtained three versions of the Landau-type theorem for bounded bi-analytic functions. Two of them are given below.
\begin{Thm}\cite{Liu-Ponnusamy-2024}\label{Thm-D}
Suppose that $\Lambda_1 \geq 0$ and $\Lambda_2>1$. Let $F(z)=\bar{z} G(z)+$ $H(z)$ be a bi-analytic function, with $G(0)=H(0)=H'(0)-1=0,|G'(z)| \leq \Lambda_1$ and $|H'(z)|<\Lambda_2$ for all $z \in \mathbb{D}$. Then $F$ is univalent on a disc $\ID_{\rho_1}$ with 
$$ \rho_1=\frac{2 \Lambda_2}{\Lambda_2\left(2 \Lambda_1+\Lambda_2\right)+\sqrt{\Lambda_2^2\left(2 \Lambda_1+\Lambda_2\right)^2-8 \Lambda_1 \Lambda_2}}
$$
 and 
 $F(\ID_{\rho_1})$ contains a schlicht disc $\ID_{\sigma_1}$ with $$ \sigma_1=F_1\left(\rho\right), \quad F_1(z)=\Lambda_2^2 z-\Lambda_1|z|^2+\left(\Lambda_2^3-\Lambda_2\right) \ln \left(1-\frac{z}{\Lambda_2}\right). $$
  This result is sharp, with an extremal function given by $F_1(z)$.
\end{Thm}
\begin{Thm}\cite{Liu-Ponnusamy-2024}\label{Thm-E}
Suppose that $\Lambda \geq 0$. Let $F(z)=\bar{z} G(z)+$ $H(z)$ be a bi-analytic function, with $G(0)=H(0)=H'(0)-1=0,|G'(z)| \leq \Lambda$  and $|H(z)|<1$ or $|H'(z)| \leq 1$ for all $z \in \mathbb{D}$.Then $F$ is univalent on a disc $\ID_{\rho_2}$ with

$$
\rho_2=\left\{\begin{array}{cl}
1, & \text { when } 0 \leq \Lambda \leq \frac{1}{2}, \vspace{2mm}\\ 
\frac{1}{2 \Lambda}, & \text { when } \Lambda>\frac{1}{2},
\end{array}\right.
$$
and $F(\ID_{\rho_1})$ contains a schlicht disc $\ID_{\sigma_1}$ with $\sigma_2=\rho_2-\Lambda \rho_2^2$. This result is sharp.
\end{Thm}

\section{Key Results}
\begin{lem}\cite{Liu-Ponnusamy-2024}\label{Rohi-Vasu-P4-Lemma-001} 
Let $H$ be a holomorphic mapping of the unit disc $\ID$ with $H'(0)=1$ and $\left|H^{\prime}(z)\right|<\Lambda$ for all $z \in \mathbb{D}$ and for some $\Lambda>1$.
\begin{enumerate}
\item  For all $z_1, z_2 \in \mathbb{D}_r\left(0<r<1, z_1 \neq z_2\right)$, we have
$$
\left|H\left(z_1\right)-H\left(z_2\right)\right|=\left|\int_\gamma H^{\prime}(z) d z\right| \geq \Lambda \frac{1-\Lambda r}{\Lambda-r}\left|z_1-z_2\right|,
$$
where $\gamma=\left[z_1, z_2\right]$ denotes the closed line segment joining $z_1$ and $z_2$.
\item  For $z^{\prime} \in \partial \mathbb{D}_r(0<r<1)$ with $w^{\prime}=H\left(z^{\prime}\right) \in H\left(\partial \mathbb{D}_r\right)$ and $\left|w^{\prime}\right|=\min \{|w|: w \in H$ $\left.\left(\partial \mathbb{D}_r\right)\right\}$, set $\gamma_0=H^{-1}\left(\Gamma_0\right)$ and $\Gamma_0=\left[0, w^{\prime}\right]$ denotes the closed line segment joining the origin and $w^{\prime}$. Then we have
$$
\left|H\left(z^{\prime}\right)\right| \geq \Lambda \int_0^r \frac{\frac{1}{\Lambda}-t}{1-\frac{t}{\Lambda}} d t=\Lambda^2 r+\left(\Lambda^3-\Lambda\right) \ln \left(1-\frac{r}{\Lambda}\right) .
$$
\end{enumerate}
\end{lem}

\begin{lem}\cite{Liu-Ponnusamy-2024}\label{Rohi-Vasu-P4-Lemma-002} 
Let $f=\sum_{n=1}^{\infty}a_nz^n$ be a holomorphic mapping of the unit disc with $|a_1|=1$ and $|f(z)|\leq M$ for each $z\in \ID$. Then $M\geq 1$ and 
$$
\left|a_n\right| \leq M-\frac{1}{M} \text { for } n=2,3, \ldots.
$$
These inequalities are sharp, with the extremal functions $f_n(z)$, where
$$
f_1(z)=z, \quad f_n(z)=M z \frac{1-M z^{n-1}}{M-z^{n-1}}=z-\left(M-\frac{1}{M}\right) z^n-\sum_{k=3}^{\infty} \frac{M^2-1}{M^{k-1}} z^{(n-1)(k-1)+1}
$$
for $n=2,3, \ldots$.

\end{lem}

\begin{Thm}\cite[Theorem 2]{Colonna-1989} \label{Thm-F}
Let $f:\ID \to \ID$ be a holomorphic function. Then 
$$\sup_{z\in \ID} \left(1-|z|^2\right)|f'(z)|\leq 1.$$  
\end{Thm}

Now, we prove the following lemma, which is vital to the proof of main results.
\begin{lem}\label{Rohi-Vasu-P4-Lemma-003} 
For $r\in [0,1)$, let
$$
\phi(r)=\frac{\Lambda(1-\Lambda r)}{\Lambda-r}-\sum_{k=1}^{m-1}r^k\left(\frac{M_k}{1-r^2}+kM_k\right)
$$
where $\Lambda> 1$ and $M_k\geq 0$ are real constants. Then $\phi$ is strictly decreasing and there is unique $r_0 \in (0,1)$ such that $\phi(r_0)=0$. Furthermore, we have
$$\Lambda^2 r_0+(\Lambda^3-\Lambda)\ln \left(1- \frac{r_0}{\Lambda}\right)-\sum_{k=1}^{m-1}r_0^{k+1}M_k>0.$$
\end{lem}

\begin{proof}
The function $\phi$ is strictly decreasing follows from the fact that
\begin{align*}
 \phi'(r)& =\frac{\Lambda-\Lambda^3}{(\Lambda-r)^2}-\sum_{k=1}^{m-1} \frac{k r^{k-1}\left(1+r^2\right)}{(1-r^2)^2}M_k+ k^2 r^{k-1} M_k \\
& =\frac{\Lambda-\Lambda^3}{(\Lambda-r)^2}-\sum_{k=1}^{m-1} k r^{k-1}\left(\frac{1+r^2}{(1-r^2)^2}M_k+k M_k\right) \\
& <0 \quad \text { for } r \in(0,1) \text {. }
\end{align*}
Furthermore, we observe that
$$
\phi(0)=1 \text { and } \lim _{r \to 1^{-}} \phi(r)=-\infty \text {. }
$$
Therefore,  by intermediate value theorem and the fact that $\phi$ is strictly decreasing, there exists unique real $r_0\in (0,1)$ such that $\phi(r_0)=0$.

\vspace{2mm}
Since $\phi$ is strictly decreasing and $\phi(0)=1$, we see that
$$
\frac{\Lambda(1-\Lambda r)}{\Lambda-r}-\sum_{k=1}^{m-1}r^k\left(\frac{M_k}{1-r^2}+kM_k\right)>0 \quad \text{for}\ r\in (0,r_0).
$$
This implies that
\begin{equation}\label{Rohi-Vasu-P4-equation-001}
\frac{\Lambda(1-\Lambda r)}{\Lambda-r}-\sum_{k=1}^{m-1}r^k(k+1)M_k >0 \quad \text{for}\ r\in (0,r_0).
\end{equation}
Integrating \eqref{Rohi-Vasu-P4-equation-001} from $0$ to $r_0$, we obtain
$$
\Lambda^2 r_0+(\Lambda^3-\Lambda)\ln \left(1- \frac{r_0}{\Lambda}\right)-\sum_{k=1}^{m-1}r_0^{k+1}M_k  >0 .
$$
This completes the proof.
\end{proof}
\section{\textbf{Landau-Bloch type Theorem for Poly-analytic functions}}
In this section, we derive three Landau-type theorems for certain subclasses of poly-analytic functions. We first establish the Landau-type theorem for subclass $\mathcal{F}_1$ of poly-analytic functions.
\begin{thm}\label{Rohi-Vasu-P4-Theorem-001}
Suppose that $m$ is a positive integer and $M_1,M_2,..,M_{k-1}\geq 0$.
Let $F(z)=\sum_{k=0}^{m-1}\bar{z}^kf_k(z)$ be a poly-analytic function of order $m$ on the unit disc $\ID$, where all the $f_k$ are holomorphic on $\ID$, satisfying $f_k(0)=F_z(0)-1=0$ for $k\in \{0,1,\cdots,m-1\}.$ If $|f_0'(z)|< \Lambda$ and $|f_k(z)|\leq M_k$, $k\in \{1,2,\cdots,m-1\}$ for all $z\in \ID$, then
$\Lambda>1$, $F(z)$ is univalent in $\ID_{r_1}$ and $F(\ID_{r_1})$ contains a schlicht disc $\ID_{R_1}$, where $r_1$ is the unique root in $(0,1)$  of the equation 
$$
\frac{\Lambda(1-\Lambda r)}{\Lambda-r}-\sum_{k=1}^{m-1}r^k\left(\frac{M_k}{1-r^2}+kM_k\right)=0
$$
and 
$$
R_1=\Lambda^2 r_1+(\Lambda^3-\Lambda)\ln \left(1- \frac{r_1}{\Lambda}\right)-\sum_{k=1}^{m-1}r_1^{k+1}M_k.
$$
Morever, when $M_k=0, k\in \{1,2,\cdots,m-1\}$, the result is sharp.
\end{thm}
\begin{proof}
We first prove that $F$ is univalent in $\ID_{r_1}$. Let $z_1,z_2 \in \ID_r (0<r<r_1)$ be two distinct points, then
\begin{align*}
 |F(z_1)-F(z_2)|&=\left|\sum_{k=0}^{m-1} \bar{z}_1^k f_k(z_1)-\sum_{k=0}^{m-1} \bar{z}_2^k f_k(z)\right| \\
& \geq|f_0(z_1)-f_0(z_2)|-\left|\sum_{k=1}^{m-1} \bar{z}_1^k f_k(z_1)-\sum_{k=1}^{m-1} \bar{z}_2^k f_k(z_2)\right| .
\end{align*}
Since $F_z(0)=f_0'(0)=1, |f_0'(z)|<\Lambda$, by Lemma \ref{Rohi-Vasu-P4-Lemma-001}, we have 
\begin{equation}\label{Rohi-Vasu-P4-equation-002}
|f_0(z_1)-f_0(z_2)|\geq \Lambda \frac{1-\Lambda r}{\Lambda-r}|z_1-z_2|.
\end{equation}
Let $[z_1,z_2]$ be the line segment joining $z_1$ to $z_2$, then by Theorem \ref{Thm-F} and the  Schwarz lemma we have
\begin{align*}
\left|\sum_{k=1}^{m-1} \bar{z}_1^k f_k(z_1)-\sum_{k=1}^{m-1} \bar{z}_2^k f_k(z_2)\right|&= \left|\sum_{k=1}^{m-1} \int_{[z_1,z_2]} \left( \bar{z}^k f_k'(z)dz+k\bar{z}^{k-1} f_k(z)d\bar{z}\right)\right|\\
& \leq \sum_{k=1}^{m-1} \int_{[z_1,z_2]} |\bar{z}|^k |f_k'(z)||dz|\\
& \quad \quad +\sum_{k=1}^{m-1} \int_{[z_1,z_2]}k|\bar{z}|^{k-1} |f_k(z)||d\bar{z}|\\
& \leq |z_1-z_2|\left(\sum_{k=1}^{m-1} r^k \frac{M_k}{1-r^2}+\sum_{k=1}^{m-1}kr^{k-1}M_k r\right)\\
&= |z_1-z_2|\left(\sum_{k=1}^{m-1}r^k\left(\frac{M_k}{1-r^2}+kM_k\right)\right).
\end{align*}
Therefore by Lemma \ref{Rohi-Vasu-P4-Lemma-003} and \eqref{Rohi-Vasu-P4-equation-002} and, we have 
\begin{align*}
|F(z_1)-F(z_2)| & \geq  |z_1-z_2|\left[\frac{\Lambda (1-\Lambda r)}{\Lambda-r}-\sum_{k=1}^{m-1}r^k\left(\frac{M_k}{1-r^2}+kM_k\right)\right]\\ 
& >  |z_1-z_2| \left[\frac{\Lambda(1-\Lambda r_1)}{\Lambda-r_1}-\sum_{k=1}^{m-1}r_1^k\left(\frac{M_k}{1-r_1^2}+kM_k\right)\right]=0.
\end{align*}
Thus $F(z_1)\neq F(z_2)$, which implies the univalence of  $F$ in $ \ID_{r_1}$.

\vspace{2mm}
To prove the second part of the theorem, let $z\in \partial \ID_{r_1},$ by Lemma \ref{Rohi-Vasu-P4-Lemma-001} and the Schwarz lemma, we have
\begin{align*}
|F(z)|&\geq |f_0(z)|-\sum_{k=1}^{m-1} |\bar{z}|^k|f_k(z)|\\
&\geq \Lambda^2 r_1 + (\Lambda^3-\Lambda)\ln \left(1-\frac{r_1}{\Lambda}\right)-\sum_{k=1}^{m-1} r_1^kM_kr_1\\
&= \Lambda^2 r_1 + (\Lambda^3-\Lambda)\ln \left(1-\frac{r_1}{\Lambda}\right)-\sum_{k=1}^{m-1} r_1^{k+1}M_k= R_1>0.
\end{align*}
Hence, $F(\ID_{r_1})$ contains a disc of positive radius $R_1$.

\vspace{2mm}
Finally, when $M_k=0, k=1,2,...,m-1,$ $F(z)=f_0(z)$ is a holomorphic mapping, it follows from \cite[Theorem 1]{Liu-Chen-2018} that the result is sharp. This completes the proof.
\end{proof}
Now, we generalize Theorem \ref{Thm-E} by establishing the following result for the subclass $\mathcal{F}_2$ of poly-analytic functions.
\begin{thm}\label{Rohi-Vasu-P4-Theorem-002}
Suppose that $m$ is a positive integer. Let $F(z)=\sum_{k=0}^{m-1}\bar{z}^kf_k(z)$ be a poly-analytic function of order $m$ on the unit disc $\ID$, where all the $f_k$ are holomorphic on $\ID$, satisfying $f_k(0)=F_z(0)-1=0$ for $k\in \{0,1,\cdots,m-1\}.$ If $|f_0(z)|< M$ and $|f_k'(z)|\leq \Lambda_k$, $k\in \{1,2,\cdots,m-1\}$ for all $z\in \ID$, then
$\Lambda_k\geq 0, M\geq 1$, $F(z)$ is univalent in $\ID_{r_2}$ and $F(\ID_{r_2})$ contains a schlicht disc $\ID_{R_2}$, where $r_2$ is the unique root in $(0,1)$ of the equation
$$
1-\left(M-\frac{1}{M}\right)\frac{2r-r^2}{(1-r)^2}-\sum_{k=1}^{m-1}(k+1)r^k\Lambda_k=0
$$
and 
$$
R_2=r_2-\left(M-\frac{1}{M}\right)\frac{r_2^2}{1-r_2}-\sum_{k=1}^{m-1}r^{k+1}\Lambda_k.
$$

\end{thm}

\begin{proof}
By the hypothesis of Theorem \ref{Rohi-Vasu-P4-Theorem-002}, for each $k\in \{1, 2,\cdots,m-1\}$, we have 
\begin{equation}\label{Rohi-Vasu-P4-equation-003}
|f_k(z)|=\left|\int_{[0,z]}f_k'(z)dz\right|\leq \int_{[0,z]}|f_k'(z)||dz|\leq \Lambda_k|z|, \quad z\in \ID.
\end{equation} 
 Let $f_0(z)=\sum_{n=1}^{\infty}a_{n,0}z^n$ be the series representation of $f_0(z)$ and $[z_1,z_2]$ denote the line segment joining $z_1$ and $z_2$. Then for any $z_1,z_2 \in \ID_r(0<r<r_2)$ with $z_1 \neq z_2$, we have 
\begin{align*}
|F(z_1)-F(z_2)|&=\left|\int_{[z_1,z_2]}F_z(z)dz+F_{\bar{z}}(z)d\bar{z}\right|\\
&=\left|\int_{[z_1,z_2]}F_z(0)dz+F_{\bar{z}}(0)d\bar{z} \right. \\
&\quad \quad +\left. \int_{[z_1,z_2]}\left(F_z(z)-F_z(0)\right)dz+\left(F_{\bar{z}}(z)-F_{\bar{z}}(0)\right)d\bar{z}\right|\\
&= \left|\int_{[z_1,z_2]}dz+ \int_{[z_1,z_2]}(f_0'(z)-1)dz \right. \\
&\quad \quad + \left. \int_{[z_1,z_2]}\sum_{k=1}^{m-1}\left[\bar{z}^kf_k'(z)dz+k\bar{z}^{k-1}f_k(z)d\bar{z}\right] \right|\\
& \geq  |z_1-z_2|- \int_{[z_1,z_2]} \sum_{n=2}^{\infty}n|a_{n,0}|r^{n-1}|dz| \\
&\quad \quad-\int_{[z_1,z_2]}\left(\sum_{k=1}^{m-1}|\bar{z}|^k|f_k'(z)|+k|\bar{z}|^{k-1}|f_k(z)|\right)|d\bar{z}|.
\end{align*}
By the equation \eqref{Rohi-Vasu-P4-equation-003} and Lemma \ref{Rohi-Vasu-P4-Lemma-002}, we have
\begin{align*}
|F(z_1)-F(z_2)|& \geq |z_1-z_2|- \int_{[z_1,z_2]}\sum_{n=2}^{\infty}n\left(M-\frac{1}{M}\right)r^{n-1}|dz|\\
& \quad \quad-\int_{[z_1,z_2]}\left(\sum_{k=1}^{m-1}\Lambda_kr^k+kr^{k-1}\Lambda_kr\right)|d\bar{z}|\\
& \geq |z_1-z_2|\left(1-\sum_{n=2}^{\infty}nr^{n-1}\left(M-\frac{1}{M}\right)-\sum_{k=1}^{m-1}(k+1)r^k\Lambda_k\right)\\
&= |z_1-z_2|\left(1-\left(M-\frac{1}{M}\right)\frac{2r-r^2}{(1-r)^2}-\sum_{k=1}^{m-1}(k+1)r^k\Lambda_k\right).
\end{align*}
For $r\in [0,1)$, consider the continuous function
$$
\psi_1(r)=1-\left(M-\frac{1}{M}\right)\frac{2r-r^2}{(1-r)^2}-\sum_{k=1}^{m-1}(k+1)r^k\Lambda_k.
$$
Similar to the proof of Lemma \ref{Rohi-Vasu-P4-Lemma-003}, we can verify that $\psi_1(r)$ is a strictly decreasing function in $[0,1]$ and there exists  unique real number $r_2\in (0,1)$ such that $\psi_1(r_2)=0$.\vspace{2mm}
Thus, we have 
$$
|F(z_1)-F(z_2)|\geq \psi_1(r)>\psi_1(r_2)=0.
$$
This implies $F(z_1)\neq F(z_2),$ which proves the univalence of $F$ in the disc $\ID_{r_2}$.

\vspace{2mm}
To prove the second part of of the theorem, let $z\in \partial \ID_{r_2},$ by  Lemma \ref{Rohi-Vasu-P4-Lemma-002} and equation \eqref{Rohi-Vasu-P4-equation-003}
\begin{align*}
|F(z)|&=\left|\sum_{k=0}^{m-1}\bar{z}^k f_k(z)\right|\\
& =\left|\sum_{n=1}^{\infty}a_{n,0}z^n+\sum_{k=1}^{m-1}\bar{z}^k f_k(z)\right|\\
 & \geq |a_{1,0}z|-\left|\sum_{n=2}^{\infty}a_{n,0}z^n\right|-\left|\sum_{k=1}^{m-1}\bar{z}^k f_k(z)\right|\\
& \geq  r_2-\sum_{n=2}^{\infty}|a_{n,0}|r_2^n-\sum_{k=1}^{m-1}r_2^k|f_k(z)|\\
& \geq r_2-\sum_{n=2}^{\infty}\left(M-\frac{1}{M}\right)r_2^n-\sum_{k=1}^{m-1}r_2^{k+1}\Lambda_k\\
&= r_2- \left(M-\frac{1}{M}\right)\frac{r_2^2}{1-r_2}-\sum_{k=1}^{m-1}r_2^{k+1}\Lambda_k=R_2.
\end{align*}
By following the same method as in the proof of Lemma \ref{Rohi-Vasu-P4-Lemma-003}, we can show that $R_2$  is strictly positive. Hence, $F(\ID_{r_2})$ contains a disc of positive radius $R_2$. This completes the proof.
\end{proof}

Finally, we obtain a sharp result for the Landau-type theorem for subclass $\mathcal{F}_3$ of poly-analytic functions, which generalizes the results of Theorem \ref{Thm-D}.
\begin{thm}\label{Rohi-Vasu-P4-Theorem-003}
Suppose that $m$ is a positive integer, $\Lambda_1, \cdots, \Lambda_{m-1} \geq 0$ and $\Lambda_0>1$. Let $F(z)=\sum_{k=0}^{m-1}\bar{z}^k f_k(z) $ be a poly-analytic function of unit disc $\ID$, where all $f_k$ are holomorphic on $\ID$, satisfying $f_k(0)=F_z(0)-1=0$ for $k\in \{0,1,\cdots,m-1\}$. If $|f_0'(z)|<\Lambda_0$ and $|f_k'(z)| \leq \Lambda_k, k \in\{1,2, \cdots, m-1\}$ for all $z \in \ID$, then $F(z)$ is univalent in $\ID_{r_3}$, and $F\left(\ID_{r_3}\right)$ contains a schlicht disc $\ID_{R_3}$, where $r_3$ is the unique root in $(0,1)$ of the following equation:
$$
\Lambda_0 \frac{1-\Lambda_0 r}{\Lambda_0-r}-\sum_{k=1}^{m-1}(k+1) \Lambda_k r^k=0
$$
and
$$
R_3=\Lambda_0^2 r_3+\left(\Lambda_0^3-\Lambda_0\right) \ln \left(1-\frac{r_3}{\Lambda_0}\right)-\sum_{k=1}^{m-1} \Lambda_k r_3^{k+1}.
$$
Moreover, these estimates are sharp, with an extremal function given by
\begin{equation}\label{Rohi-Vasu-P4-equation-004}
F_0(z)=\Lambda_0^2 z+\left(\Lambda_0^3-\Lambda_0\right) \ln \left(1-\frac{z}{\Lambda_0}\right)-\sum_{k=1}^{m-1} \bar{z}^k \Lambda_k z, z \in \ID.
\end{equation}
\end{thm}
\begin{proof}
We first prove that $F$ is univalent in $\ID_{r_3}$. Let $z_1,z_2 \in \ID_r (0<r<r_3)$ be two distinct points, then
\begin{align*}
 |F(z_1)-F(z_2)|&=\left|\sum_{k=0}^{m-1} \bar{z}_1^k f_k(z_1)-\sum_{k=0}^{m-1} \bar{z}_2^k f_k(z)\right| \\
& \geq|f_0(z_1)-f_0(z_2)|-\left|\sum_{k=1}^{m-1} \bar{z}_1^k f_k(z_1)-\sum_{k=1}^{m-1} \bar{z}_2^k f_k(z_2)\right| .
\end{align*}
Since $F_z(0)=f_0'(0)=1, |f_0'(z)|<\Lambda_0$, in view of Lemma \ref{Rohi-Vasu-P4-Lemma-001}, we have 
\begin{equation*}
|f_0(z_1)-f_0(z_2)|\geq \Lambda_0 \frac{1-\Lambda_0 r}{\Lambda_0-r}|z_1-z_2|.
\end{equation*}
Since $|f_k'(z)|\leq \Lambda_k$ for each $k\in \{1,2,,\cdots,m-1\}$, we have from \eqref{Rohi-Vasu-P4-equation-003} that $|f_k(z)|\leq \Lambda_k|z|,\  z\in \ID.$
Let $[z_1,z_2]$ denote the line segment joining $z_1$ and $z_2$, then we have 
\begin{align*}\label{Rohi-Vasu-P4-equation-005}
\left|\sum_{k=1}^{m-1} \bar{z}_1^k f_k(z_1)-\sum_{k=1}^{m-1} \bar{z}_2^k f_k(z_2)\right|&= \left|\sum_{k=1}^{m-1} \int_{[z_1,z_2]} \left( \bar{z}^k f_k'(z)dz+k\bar{z}^{k-1} f_k(z)d\bar{z}\right)\right|\\
& \leq \sum_{k=1}^{m-1} \int_{[z_1,z_2]} |\bar{z}|^k |f_k'(z)||dz|\\
& \quad \quad +\sum_{k=1}^{m-1} \int_{[z_1,z_2]}k|\bar{z}|^{k-1} |f_k(z)||d\bar{z}|\\
& \leq \sum_{k=1}^{m-1} \int_{[z_1,z_2]}r^k \Lambda_k |dz|+\sum_{k=1}^{m-1} \int_{[z_1,z_2]}kr^{k-1}\Lambda_k r |d\bar{z}|\\
&= \left(\sum_{k=1}^{m-1}(k+1)\Lambda_kr^k\right)|z_1-z_2|.\numberthis
\end{align*}
Therefore, we have
$$|F(z_1)-F(z_2)|\geq \left(\frac{\Lambda_0(1-\Lambda_0 r)}{\Lambda_0-r}-\sum_{k=1}^{m-1}(k+1)\Lambda_kr^k \right)|z_1-z_2|.$$
For $r\in [0,1]$, consider the following continuous function
$$
\psi_2(r)=\frac{\Lambda_0(1-\Lambda_0 r)}{\Lambda_0-r}-\sum_{k=1}^{m-1}(k+1)\Lambda_kr^k.
$$
Similar to the proof of Lemma \ref{Rohi-Vasu-P4-Lemma-003}, we can verify that $\psi_2(r)$ is a strictly decreasing function in $[0,1]$. Therefore, there exists unique real number $r_3\in (0,1)$ such that $\psi_2(r_3)=0$. Thus, we have 
$$
|F(z_1)-F(z_2)|\geq \psi_2(r)>\psi_2(r_3)=0.
$$
This implies $F(z_1)\neq F(z_2),$ which proves the univalence of $F$ in the disc $\ID_{r_3}$.

\vspace{2mm}
Next, for any $z\in \partial \ID_{r_3}$ by Lemma \ref{Rohi-Vasu-P4-Lemma-001} and equation \eqref{Rohi-Vasu-P4-equation-003}, we have
\begin{align*}
|F(z)|&=\left|\sum_{k=0}^{m-1}\bar{z}^k f_k(z)\right|\\
&\geq |f_0(z)|-\sum_{k=1}^{m-1}|\bar{z}|^k |f_k(z)|\\
& \geq \Lambda_0^2r_3+(\Lambda_0^3-\Lambda_0)\ln\left(1-\frac{r_3}{\Lambda_0}\right)-\sum_{k=1}^{m-1}r_3^{k+1}\Lambda_k=R_3.
\end{align*}
We can easily see that the quantity $R_3$ is strictly positive. This shows that $ F(\ID_{r_3})$ contains a Schlicht disc $\ID_{R_3}$.

\vspace{2mm}
Now, to prove the sharpness of $r_3$ and $R_3$, we consider a poly-analytic function $F_0(z)$ given by \eqref{Rohi-Vasu-P4-equation-004}. It is easy to see that $F_0(z)$ satisfies all the conditions of the Theorem \ref{Rohi-Vasu-P4-Theorem-003}. Hence $F_0(z)$ is univalent in $\ID_{r_3}$ and $F_0(\ID_{r_3})\supseteq \ID_{R_3}$. 

\vspace{2mm}
To prove the univalent radius $r_3$ is sharp, we need to prove that $F_0(z)$ is not univalent in $\ID_r$ for each $r\in (r_3,1]$. Indeed, let
\begin{equation*}
g(x)=\Lambda_0^2x+(\Lambda_0^3-\Lambda_0)\ln\left(1-\frac{x}{\Lambda_0}\right)-\sum_{k=1}^{m-1}x^{k+1}\Lambda_k, \quad x\in [0,1].
\end{equation*}
Since, the continuous function
 \begin{equation*}
  g'(x)=\frac{\Lambda_0(1-\Lambda_0 x)}{\Lambda_0-x}-\sum_{k=1}^{m-1}(k+1)\Lambda_kx^k=\psi_2(x) 
  \end{equation*} 
 is strictly decreasing on $[0,1]$ and $g'(r_3)=\psi_2(r_3)=0$, we obtain that $g'(x)=0$ if, and only if, $x=r_3.$ Therefore, $g(x)$ is strictly increasing on $[0,r_3]$ and strictly decreasing on $[r_3,1]$. We see that if $g(1)\leq 0$, then there is unique real $r'\in (r_3,1] $ such that $g(r')=0.$

\vspace{2mm}
For every fixed $r\in (r_3,1]$, we choose
\begin{equation*}
\varepsilon=\begin{cases} \min \left\{\dfrac{r-r_3}{2}, \dfrac{r'-r_3}{2}\right\}, & \text { if } g(1)\leq 0, \vspace{2mm}\\ \dfrac{r-r_3}{2}, & \text { if } g(1)>0 . \end{cases}
\end{equation*}
For such $\varepsilon$, we let $x_1=r_3+\varepsilon$ and there exists unique $\delta\in (0, r_3)$ such that $x_2:=r_3-\delta\in (0,r_3)$ and $g(x_1)=g(x_2).$ By setting $z_1=x_1$ and $z_2=x_2$, we see that $z_2,z_2\in \ID_r$ with $z_2\neq z_2$ and 
$$
F_0(z_1)=F_0(x_1)=g(x_1)=g(x_2)=F_0(x_2)=F_0(z_2).
$$ 
This implies that $F_0$ is not univalent in $\ID_r$ for each $r\in (r_3,1]$, which proves the sharpness of $r_3$.

\vspace{2mm}
 Finally, to prove the sharpness of $R_3$, we take $z'=r_3\in \partial\ID_{r_3}$, then we have
 $$
 |F_0(z')-F_0(0)|=|F_0(r_3)|=g(r_3)=R_3.
 $$
Hence, the radius $R_3$ of the schlicht disc is also sharp. This completes the proof.
\end{proof}

\section{\textbf{The Bi-Lipschitz Theroems for  Some Poly-analytic functions}}
In this section, we investigate the Lipschitz character of subclasses $\mathcal{F}_1, \mathcal{F}_2$ and $\mathcal{F}_3$ of poly-analytic functions on their univalent discs.

\begin{thm}\label{Rohi-Vasu-P4-Theorem-004}
Suppose $F(z)$ satisfies the hypothesis of Theorem \ref{Rohi-Vasu-P4-Theorem-001}. Then for each $\rho_1 \in (0, r_1)$, the poly-analytic function $F(z)$ is bi-Lipschitz on $\overline{\ID}_{\rho_1}$, where $r_1$ is given by Theorem \ref{Rohi-Vasu-P4-Theorem-001}.
\end{thm}
\begin{proof}
Fix $\rho_1 \in (0,r_1)$ and let
$$
l_1=\frac{\Lambda(1-\Lambda \rho_1)}{\Lambda-\rho_1}-\sum_{k=1}^{m-1}\rho_1^k\left(\frac{M_k}{1-\rho_1^2}+kM_k\right).
$$
Then for any $z_1, z_2 \in \overline{\ID}_{\rho_1}$, it follows from the proof of Theorem \ref{Rohi-Vasu-P4-Theorem-001} and $|f_0'(z)|< \Lambda$ for $z \in \ID$ that $l_1>0$ and 
\begin{align*}
l_1|z_1-z_2|&=\left(\frac{\Lambda(1-\Lambda \rho_1)}{\Lambda-\rho_1}-\sum_{k=1}^{m-1}\rho_1^k\left(\frac{M_k}{1-\rho_1^2}+kM_k\right)\right)|z_1-z_2|\\
& \leq |F(z_1)-F(z_2)|\\
& \leq |f_0(z_1)-f_0(z_2)|+\left|\sum_{k=1}^{m-1} \bar{z}_1^k f_k(z_1)-\sum_{k=1}^{m-1} \bar{z}_2^k f_k(z_2)\right|\\
 & \leq \left|\int_{[z_1,z_2]}f_0'(z)dz\right|+|z_1-z_2|\sum_{k=1}^{m-1}\rho_1^k\left(\frac{M_k}{1-\rho_1^2}+kM_k\right)\\
 & \leq \left( \Lambda + \sum_{k=1}^{m-1}\rho_1^k\left(\frac{M_k}{1-\rho_1^2}+kM_k\right)\right) |z_1-z_2|.
\end{align*}
Hence $F(z)$ is bi-Lipschitz on $\ID_{\rho_1}$.
\end{proof}

\begin{thm}\label{Rohi-Vasu-P4-Theorem-005}
Suppose $F(z)$ satisfies the hypothesis of Theorem \ref{Rohi-Vasu-P4-Theorem-002}. Then for each $\rho_2 \in (0, r_2)$, the poly-analytic function $F(z)$ is bi-Lipschitz on $\overline{\ID}_{\rho_2}$, \it{i.e.}, for any $z_1,z_2 \in \overline{\ID}_{\rho_2}$, there exist $$ l_2=1-\left(M-\frac{1}{M}\right)\frac{2\rho_2-\rho_2^2}{(1-\rho_2)^2}-\sum_{k=1}^{m-1}(k+1)\rho_2^k\Lambda_k>0$$ and 
$$L_2= \frac{M}{1-\rho_2^2}+\sum_{k=1}^{m-1}\rho_2^k(k+1)\Lambda_k$$ such that 
$$l_2|z_1-z_2|\leq |F(z_1)-F(z_2)|\leq L_2|z_1-z_2|,$$
where $r_2$ is given by Theorem \ref{Rohi-Vasu-P4-Theorem-002}.
\end{thm}
\begin{proof}
Fix $\rho_2 \in (0,r_2)$. Then for any $z_1, z_2 \in \overline{\ID}_{\rho_2}$, it follows from the proof of Theorem \ref{Rohi-Vasu-P4-Theorem-002} that $l_2>0$ and 
\begin{align*}
l_2|z_1-z_2| & \leq |F(z_1)-F(z_2)|\vspace{2mm}\\ 
& \leq |f_0(z_1)-f_0(z_2)|+\left|\sum_{k=1}^{m-1} \bar{z}_1^k f_k(z_1)-\sum_{k=1}^{m-1} \bar{z}_2^k f_k(z_2)\right| \vspace{2mm}\\
& = \left|\int_{[z_1,z_2]}f_0'(z)dz\right|+\left|\sum_{k=1}^{m-1} \bar{z}_1^k f_k(z_1)-\sum_{k=1}^{m-1} \bar{z}_2^k f_k(z_2)\right|
\end{align*}
Now since $|f_0(z)|<M$ and $|f_k'(z)|\leq \Lambda_k$, by Theorem \ref{Thm-F} and \eqref{Rohi-Vasu-P4-equation-005} we have
\begin{align*}
l_2|z_1-z_2| & \leq |F(z_1)-F(z_2)|\\ 
& \leq \left(\frac{M}{1-\rho_2^2}+\sum_{k=1}^{m-1}\rho_2^k(k+1)\Lambda_k\right)|z_1-z_2|=L_2|z_1-z_2|.
\end{align*}
This completes the proof.

\end{proof}
\begin{thm}\label{Rohi-Vasu-P4-Theorem-006}
Suppose $F(z)$ satisfies the hypothesis of Theorem \ref{Rohi-Vasu-P4-Theorem-003}. Then for each $\rho_3 \in (0, r_3)$, the poly-analytic function $F(z)$ is bi-Lipschitz on $\overline{\ID}_{\rho_3}$, \it{i.e.}, for any $z_1,z_2 \in \overline{\ID}_{\rho_3}$, there exist $$ l_3=\frac{\Lambda_0(1-\Lambda_0 \rho_3)}{\Lambda_0-\rho_3}-\sum_{k=1}^{m-1}(k+1)\Lambda_k\rho_3^k >0$$ and 
$$L_3= \Lambda+\sum_{k=1}^{m-1}\rho_3^k(k+1)\Lambda_k$$ such that 
$$l_3|z_1-z_2|\leq |F(z_1)-F(z_2)|\leq L_3|z_1-z_2|,$$
where $r_3$ is given by Theorem \ref{Rohi-Vasu-P4-Theorem-003}.
\end{thm}
\begin{proof}
Fix $\rho_3 \in (0,r_3)$. Then for any $z_1, z_2 \in \overline{\ID}_{\rho_3}$, it follows from the proof of Theorem \ref{Rohi-Vasu-P4-Theorem-003} and $|f_0'(z)|< \Lambda_0$ for $z \in \ID$ that $l_3>0$ and 
\begin{align*}
l_3|z_1-z_2|& \leq |F(z_1)-F(z_2)|\\
& \leq |f_0(z_1)-f_0(z_2)|+\left|\sum_{k=1}^{m-1} \bar{z}_1^k f_k(z_1)-\sum_{k=1}^{m-1} \bar{z}_2^k f_k(z_2)\right|\\
 & \leq \left|\int_{[z_1,z_2]}f_0'(z)dz\right|+|z_1-z_2|\sum_{k=1}^{m-1}\rho_3^k(k+1)\Lambda_k\\
 & \leq \left( \Lambda + \sum_{k=1}^{m-1}\rho_3^k(k+1)\Lambda_k\right) |z_1-z_2|=L_3|z_1-z_2|.
\end{align*}
This completes the proof.

\end{proof}

\noindent\textbf{Acknowledgement:}  The second named author thank CSIR for the support. 
\vspace{1.5mm}

\noindent\textbf{Compliance of Ethical Standards:}\\
\noindent\textbf{Conflict of interest.} The authors declare that there is no conflict  of interest regarding the publication of this paper.
\vspace{1.5mm}

\noindent\textbf{Data availability statement.}  Data sharing is not applicable to this article as no datasets were generated or analyzed during the current study.\vspace{1.5mm}

\noindent\textbf{Authors contributions.} Both the authors have made equal contributions in reading, writing, and preparing the manuscript.


\begin{thebibliography}{99}

 \bibitem{Abdulhadi-Hajj-2022} {\sc Z. Abdulhadi} and {\sc L.E. Hajj}, On the univalence of poly-analytic functions, {\it Comput. Methods Funct. Theory} {\bf 22} (2022), 169--181. 
 
 \bibitem{Abdulhadi-Muhanna-2008} {\sc Z. Abdulhadi} and {\sc Y. Abu Muhanna}, Landau's theorem for biharmonic mappings, {\it J. Math. Anal. Appl.} {\bf 338} (2008), 705--709. 
 
 \bibitem{Abdulhadi-Muhanna-Ali-2012} {\sc Z. Abdulhadi}, {\sc Y. Abu Muhana} and {\sc R.M. Ali}, Landau’s theorem for functions with logharmonic Laplacian, {\it Appl. Math. Comput.} {\bf 218} (2012), 6798--6802.

\bibitem{Abreu-2010}{\sc L.D. Abreu}, Sampling and interpolation in Bargmann-Fock spaces of polyanalytic functions, {\it   Appl. Comput. Harmon. Anal.} {\bf 29}(3) (2010), 287--302.

\bibitem{Agranovsky-2011}{\sc M.L. Agranovsky}, Characterization of polyanalytic functions by meromorphic extensions from chains of circles, {\it  J. Anal. Math.} {\bf 113}(1) (2011), 305--329.

 \bibitem{Ahern-Bruna-1988} {\sc P. Ahern} and {\sc J. Bruna}, Maximal and area integral characterizations of Hardy-Sobolev spaces in the unit ball of $\IC^n$, {\it Rev. Mat. Iberoam.} {\bf 4}(1) (1988), 123--153. 


\bibitem{Allu-Halder-2023}{\sc V. Allu} and {\sc H. Halder}, Bohr operator on operator-valued polyanalytic functions on simply connected domains, {\it  Canad. Math. Bull.}{\bf 66}(4) (2023), 1411--1422. 
 
\bibitem{Allu-Kumar-2024}{\sc V. Allu} and {\sc R. Kumar}, Landau-Bloch type theorem for elliptic and quasiregular harmonic mappings, {\it  J. Math. Anal. Appl. }{\bf 535} (2024), 128215. 

\bibitem{Allu-Kumar-2024-a}{\sc V. Allu} and {\sc R. Kumar}, The Landau-Bloch type theorems for certain class of holomorphic and pluriharmonic mappings in $\IC^n$, {\it arXiv:2308.15913}. 

\bibitem{Allu-Kumar-2024-b}{\sc V. Allu} and {\sc R. Kumar}, Landau type theorem for $\alpha$-harmonic mappings, {\it arXiv:2406.01709}.

\bibitem{Bai-Liu-2019} {\sc X.X. Bai} and {\sc M.S. Liu}, Landau-Type Theorems of Polyharmonic Mappings and log-p-Harmonic Mappings, {\it Complex Anal. Oper. Theory} {\bf 13} (2019), 321--340. 

\bibitem{Balk-1991}{\sc M. Balk}, Polyanalytic Functions, Berlin, {\it  Akademie-Verlag} (1991).

\bibitem{Balk-1997}{\sc M. Balk}, Polyanalytic functions and their generalizations, Complex Analysis I Encyclopedia Math. Sci., {\it  Springer, Berlin} (1997), 195--253.



\bibitem{Chen-Gauthier-Hengertner-2000} {\sc S. Chen}, {\sc P.M. Gauthier} and {\sc W. Hengartner}, Bloch constants for planer harmonic mappings, {\it Proc. Amer. Math. Soc.} {\bf 128} (2000), 3231--3240.
 




\bibitem{Chen-Ponnusamy-Wang-2011} {\sc S. Chen}, {\sc S. Ponnusamy} and {\sc X. Wang}, Coefficient estimates and Landau–Bloch’s theorem for planar harmonic mappings, {\it Bull. Malays. Math. Sci. Soc.} {\bf 34}(2) (2011), 255--265.


\bibitem{Colonna-1989} {\sc F. Colonna}, The Bloch constant of bounded harmonic mappings, {\it  Indiana Univ. Math. J.} {\bf 38} (1989), 829--840. 

 
 \bibitem{Daghighi-Krantz-2016}{\sc A. Daghighi} and {\sc S.G. Krantz}, Local Maximum Modulus Property for Polyanalytic Functions, {\it Complex Anal. Oper. Theory} {\bf 10} (2016), 401--408.

\bibitem{Dorff-Nowak-2004}{\sc M. Dorff} and {\sc M. Nowak}, Landau's theorem for planar harmonic mappings, {\it Comput. Methods Funct. Theory} {\bf 4 } (2004), 151--158.

\bibitem{Grigoyan-2006} {\sc A. Grigoyan}, Landau and Bloch theorems for harmonic mappings, {\it  Complex Var. Elliptic Equ.} {\bf 51}(1) (2006), 81--87. 

\bibitem{Huang-2008} {\sc  X. Z. Huang}, Estimates on Bloch constants for planar harmonic mappings, {\it  J. Math. Anal. Appl.} {\bf 337} (2008), 880--887.

 \bibitem{Kolossov-1908} {\sc G.V. Kolossoff}, Sur les problèmes d’élasticité à deux dimensions {\it CR Acad. Sci. } {\bf 148}(1908), 1242--1244.

  \bibitem{Landau-1926} {\sc E. Landau}, Der Picard-Schottkysche Satz und die Blochsche Konstanten, {\it  Sitzungsber. Preuss. Akad. Wiss. Berlin Phys.-Math. Kl.} (1926), 467--474.
  
  \bibitem{Liu-2008} {\sc M.S. Liu}, Landau’s theorems for biharmonic mappings, {\it  Complex Var. Elliptic Equ. } {\bf 53}(9) (2008), 843--855.

\bibitem{M. Liu-2009} {\sc M. S. Liu}, Landau's theorems for planer harmonic mappings, {\it  Comput. Math. Appl.} {\bf 57}(7) (2009), 1142--1146.

\bibitem{Liu-2009} {\sc M.S. Liu}, Estimates on Bloch constants for planar harmonic mappings, {\it  Sci. China Ser. A-Math. } {\bf 52}(1) (2009), 87--92.

 \bibitem{Liu-Chen-2018} {\sc M.S. Liu } and {\sc H.H. Chen}, The Landau–Bloch type theorems for planar harmonic mappings with bounded dilation, {\it  J. Math. Anal. Appl.} {\bf 468}(2) (2018), 1066--1081.

\bibitem{Liu-Ponnusamy-2024} {\sc M.S. Liu} and {\sc S. Ponnusamy}, Landau-type theorems for certain bounded bi-analytic functions and biharmonic mappings, {\it Canad. Math. Bull.} {\bf 67}(1) (2024), 152--165.

\bibitem{Liu-Luo-2021} {\sc M.S. Liu} and {\sc L. Luo}, Precise values of the Bloch constants of certain log-p-harmonic mappings, {\it Acta Math. Sci. Ser. B (Engl. Ed.)} {\bf 41} (2021), 297--310.

\bibitem{Luo-Liu-2023} {\sc X. Luo} and {\sc M.S. Liu}, Landau–Bloch Type Theorems for Certain Subclasses for Polyharmonic Mappings, {\it Comput. Methods Funct. Theory} {\bf 23} (2023), 303--325.


 


 
 
 


\bibitem{Vasilevski-1999} {\sc N.L. Vasilevski}, On the structure of Bergman and poly-Bergman spaces, {\it Integral Equations Operator Theory} {\bf 33}(4) (1999), 471--488. 

\bibitem{Vasilevski-2023} {\sc N.L. Vasilevski}, On polyanalytic functions in several complex variables, {\it Complex Anal. Oper. Theory} {\bf 17}(80) (2023). 

\bibitem{Zhu-2015} {\sc J.F. Zhu}, Landau theorem for planar harmonic mappings, {\it Complex Anal. Oper. Theory} {\bf 9} (2015), 1819--1826. 



\end{thebibliography}
\end{document}